\documentclass{amsart}[12pt]
\def\doctype{}

\usepackage{latexsym,amssymb}
\usepackage{fancyhdr}
\usepackage{lineno}
\usepackage{hyperref}

\newcommand\lam{\lambda}

\newcommand{\cA}{\mathcal{A}}
\newcommand{\cB}{\mathcal{B}}

\newcommand\Z{\mathbb{Z}}
\newcommand{\comment}[1]{}

\numberwithin{equation}{section}

\fancypagestyle{titlepage}{
\fancyhead[R]{\doctype}
\fancyhead[CL]{}
\cfoot{\vspace{5pt} \thepage}
}

\let\oldsection\section
\newcommand\boldsection[1]{\oldsection{\bf #1}}
\newcommand\starsection[1]{\oldsection*{\bf #1}}
\makeatletter
\renewcommand\section{\@ifstar\starsection\boldsection}
\makeatother

\newtheoremstyle{theorem}
  {12pt}		  
  {0pt}  
  {\sl}  
  {\parindent}     
  {\bf}  
  {. }    
  { }    
  {}     
\theoremstyle{theorem}
\newtheorem{thm}{Theorem}[section]  
\newtheorem{lemma}[thm]{Lemma}     
\newtheorem{cor}[thm]{Corollary}

\newtheorem{cons}[thm]{Construction}
\newtheorem{prop}[thm]{Proposition}

\newtheoremstyle{definition}
  {12pt}		  
  {0pt}  
  {}  
  {\parindent}     
  {\bf}  
  {. }    
  { }    
  {}     
\theoremstyle{definition}

\newtheorem{ex}[thm]{Example}

\newcommand\rk{{\sc Remark.} }

\renewcommand{\proofname}{Proof}

\makeatletter
\renewenvironment{proof}[1][\proofname]{\par
  \pushQED{\qed}%
  \normalfont \partopsep=\z@skip \topsep=\z@skip
  \trivlist
  \item[\hskip\labelsep
        \scshape
    #1\@addpunct{.}]\ignorespaces
}{%
  \popQED\endtrivlist\@endpefalse
}
\makeatother

\makeatletter
\renewcommand*\@maketitle{%
  \normalfont\normalsize
  \@adminfootnotes
  \@mkboth{\@nx\shortauthors}{\@nx\shorttitle}%
  \global\topskip42\p@\relax 
  \@settitle
  \ifx\@empty\authors \else {\vskip 1em
\vtop{\centering\shortauthors\@@par}} \fi
  \ifx\@empty\@date \else {\vskip 1em \vtop{\centering\@date\@@par}}\fi 
  \ifx\@empty\@dedicatory
  \else
    \baselineskip18\p@
    \vtop{\centering{\footnotesize\itshape\@dedicatory\@@par}%
      \global\dimen@i\prevdepth}\prevdepth\dimen@i
  \fi
  \@setabstract
  \normalsize
  \if@titlepage
    \newpage
  \else
    \dimen@34\p@ \advance\dimen@-\baselineskip
    \vskip\dimen@\relax
  \fi
} 
\renewcommand*\@adminfootnotes{%
  \let\@makefnmark\relax  \let\@thefnmark\relax
  \ifx\@empty\@subjclass\else \@footnotetext{\@setsubjclass}\fi
  \ifx\@empty\@keywords\else \@footnotetext{\@setkeywords}\fi
  \ifx\@empty\thankses\else \@footnotetext{%
    \def\par{\let\par\@par}\@setthanks}%
  \fi
\thispagestyle{titlepage}
}
\makeatother

\begin{document}

\title[Incomplete designs]{\large Constructions and uses of incomplete
pairwise balanced designs}

\author{Peter J.~Dukes}
\address{\rm Peter J.~ Dukes:
Mathematics and Statistics,
University of Victoria, Victoria, Canada
}
\email{dukes@uvic.ca}

\author{Esther R.~Lamken}
\address{\rm Esther R.~Lamken:
773 Colby Street, San Francisco, CA, USA 94134
}
\email{lamken@caltech.edu}

\thanks{Research of Peter Dukes is supported by NSERC}

\date{\today}

\begin{abstract}
We give explicit constructions for incomplete pairwise balanced designs IPBD$((v;w),K)$, or, equivalently, edge-decompositions of a difference of two cliques $K_v \setminus K_w$ into cliques whose sizes belong to the set $K$.   Our constructions produce such designs whenever $v$ and $w$ satisfy the usual divisibility conditions, have ratio $v/w$ bounded away from the smallest value in $K$ minus one, say $v/w > k-1+\epsilon$, for $k =\min K$ and $\epsilon>0$, and are sufficiently large (depending on $K$ and $\epsilon$).
As a consequence, some new results are obtained on many related designs, including class-uniformly resolvable designs, incomplete mutually orthogonal latin squares, and group divisible designs. We also include several other applications that illustrate the power of using IPBDs as `templates'.

\end{abstract}
\maketitle
\hrule

\section{Introduction}
\label{intro}

\subsection{Designs and decompositions}

Let $v$ be a positive integer and $K \subseteq \Z_{\ge 2} :=
\{2,3,4,\dots\}$.  A \emph{pairwise balanced design}
PBD$(v,K)$ is a pair $(V,\cB)$, where
\begin{itemize}
\item
$V$ is a $v$-element set of \emph{points};
\item
$\cB \subseteq \cup_{k \in K} \binom{V}{k}$ is a family of of subsets of
$V$, called \emph{blocks}; and
\item
every two distinct points appear together in exactly one block.
\end{itemize}
In alternative language, a PBD$(v,K)$ is an edge-decomposition of the complete graph of order $v$ into
cliques whose sizes come from the set $K$.  A PBD$(v,K)$ with $K=\{k\}$ is also known as a Steiner system S$(2,k,v)$
or a balanced incomplete block design, $(v,k,1)$-BIBD.  Pairwise
balanced designs (and BIBDs in particular) permit an
additional parameter $\lam$ and ask that any two distinct points appear
together in exactly $\lam$ blocks.  For the moment, though, our attention is
restricted to $\lam=1$.

There are necessary divisibility conditions for existence of PBD$(v,K)$.  In what follows, put $\alpha(K):=\gcd\{k-1: k \in K\}$ and $\beta(K):=\gcd\{k(k-1): k \in K\}$.

\begin{prop}
\label{neccond}
The existence of a PBD$(v,K)$ implies
\begin{eqnarray}
\label{local-pbd}
v-1 &\equiv& 0 \pmod{\alpha(K)} ~\text{and}\\
\label{global-pbd}
v(v-1)  &\equiv& 0 \pmod{\beta(K)}.
\end{eqnarray}
\end{prop}

\begin{proof}
The $\binom{v}{2}$ pairs of points are partitioned into $\binom{k}{2}$ pairs
within each block of size $k \in K$.  It follows that $v(v-1)$ is an integer
linear combination of $k(k-1)$, $k \in K$.  This proves (\ref{global-pbd}).

Given any point $x \in V$, let $\mathcal{B}_x:= \{B \in \mathcal{B}: x \in
B\}$.  The points in $V\setminus \{x\}$ are partitioned by $B \setminus
\{x\}$, $B \in \mathcal{B}_x$.  So $v-1$ is an integer linear combination of
$k-1$, $k \in K$.  This proves  (\ref{local-pbd}).
\end{proof}

The `asymptotic sufficiency' of these conditions is a celebrated result due
to Richard M.~Wilson.

\begin{thm}[Wilson, \cite{RMW1}]
\label{asymptotic-bd}
There exist PBD$(v,K)$ for all sufficiently large $v$ satisfying the divisibility conditions
$(\ref{local-pbd})$ and $(\ref{global-pbd})$.
\end{thm}

Theorem~\ref{asymptotic-bd} lays the foundation for a rich existence theory
for combinatorial designs that fit the general framework of decomposing complete graphs; see for instance
\cite{Draganova,DL,LW,Wilson75}.  These uses include designs with extra conditions, 
such as resolvable designs and designs with automorphisms or tournament-style conditions, and
even include apparent generalizations, such as group-divisible designs and graph decompositions.

Recently, there has been exciting progress on decompositions of dense or quasi-random (hyper)graphs into prescribed
cliques (or graphs, in general); see \cite{BKLO,GKLO,Keevash,Keevash2}.  These settle existence of many types of
(extremely large) designs, even $t$-designs.  One drawback is that explicit constructions are practically lost in this setting, 
either by demanding a large pre-structure (absorbers) or repeatedly `repacking' blocks (randomized construction).

In this article, we develop explicit constructions suitable for a specific `boundary case', namely the decomposition
of a difference of cliques, and point out applications to various related designs.

\subsection{Incomplete designs}
Let $v \ge w$ be positive integers and $K\subseteq \Z_{\ge 2}$.  An
\emph{incomplete pairwise balanced design} IPBD$((v;w),K)$ is a triple
$(V,W,\cB)$ where
\begin{itemize}
\item
$V$ is a set of $v$ points and $W \subset V$ is a \emph{hole} of size $w$;
\item
$\cB \subseteq \cup_{k \in K} \binom{V}{k}$ is a family of blocks;
\item
no two distinct points of $W$ appear in a block; and
\item
every two distinct points not both in $W$ appear together in exactly one
block.
\end{itemize}
An equivalent object is a PBD$(v,K \cup \{w^*\})$, where the star indicates
that there is exactly one block of size $w$ if $w \not\in K$ and at least
one block of size $w$ if $w \in K$.  Given an IPBD$((v;w),K)$, say
$(V,W,\cB)$, the system $(V,\cB \cup \{W\})$ is a PBD$(v,K\cup \{w^*\})$.

Another closely related notion is that of a PBD$(v,K)$, say $(V,\cB)$, containing
a \emph{subdesign} PBD$(w,K)$, say $(W,\cA)$, where we have $W \subseteq V$
and $\cA \subseteq \cB$.  We obtain an IPBD$((v;w),K)$ as $(V,W,\cB
\setminus \cA)$.  On the other hand, an IPBD with hole $W$ can be `filled'
with a PBD (or another IPBD) on $W$, but only when this smaller design
exists.

The case $w=v$ leads to $\cB=\emptyset$ and we exclude this in what follows.
The case $w=1$ reduces to a PBD$(v,K)$, since such a hole contains no pairs.
By analogy with (\ref{local-pbd}) and (\ref{global-pbd}), there are natural
divisibility conditions on the parameters.

\begin{prop}
\label{neccond-h}
The existence of an IPBD$((v;w),K)$ implies
\begin{eqnarray}
\label{local}
v-1~\equiv~w-1 &\equiv& 0 \pmod{\alpha(K)}, ~\text{and}\\
\label{global}
v(v-1) - w(w-1) &\equiv& 0 \pmod{\beta(K)}.
\end{eqnarray}
\end{prop}

In \cite{DvB}, an existence result was presented for fixed $w$ and large $v$.

\begin{thm}[\cite{DvB}]
\label{fixed}
Let $K \subseteq \Z_{\ge 2}$ and let $w \equiv 1 \pmod{\alpha(K)}$.  
There exists IPBD$((v;w),K)$ for all sufficiently large $v$ satisfying 
$(\ref{local})$ and $(\ref{global})$.
\end{thm}

There is another necessary condition taking the form of an inequality.

\begin{prop}
\label{holesize-bound}
Let $k=\min K$.  Every point in $V \setminus W$ is incident with at most
$\frac{v-1}{k-1}-w$ blocks disjoint from $W$.  Therefore, the existence of
an IPBD$((v;w),K)$ with $v>w$ implies
\begin{equation}
\label{ineq-h}
v \ge (k-1)w +1.
\end{equation}
\end{prop}

\begin{proof}
Since $k$ is the smallest block size, a point $x \in V \setminus W$ is in at
most $\frac{v-1}{k-1}$ blocks.  Each such block covers at most one point in
$W$, and conversely every point in $W$ is together in exactly one block with
$x$.  It follows that, of the blocks containing $x$, exactly $w$ of them
intersect $W$.
\end{proof}

{\sc Remark.}
Equality in (\ref{ineq-h}) holds if and only if every block intersects the
hole and has size $k$.

In the case of a single block size, $K=\{k\}$, the authors and Alan C.H.~Ling settled in \cite{DLLholes} the
existence of incomplete designs in two cases: (1) fixed $w$ and large $v$; and (2) $v/w$ bounded away 
from $k-1$ for large $v$ and $w$.
Taken together, these results can be summarized as follows.

\begin{thm}[\cite{DLLholes}]
\label{ipbdk}
Let $k$ be an integer, $k \ge 2$, and $\epsilon>0$.  For some $v_0=v_0(k,\epsilon)$,
an IPBD$((v;w),\{k\})$ exists for all $v \ge v_0$ and $w$ satisfying 
the divisibility conditions and $v > (k-1+\epsilon) w$.
\end{thm}


Our aim in this paper is to generalize Theorem~\ref{ipbdk} to the general case of mixed block sizes in $K$.
There are some technical challenges here, since the divisibility conditions
weaken in general and yet the necessary inequality Proposition~\ref{holesize-bound} remains in terms of $k = \min K$
alone.

We state our main result.

\begin{thm}
\label{main}
Let $K \subseteq \Z_{\ge 2}$ with $k=\min K$ and let $\epsilon>0$.  For some $v_0=v_0(K,\epsilon)$,
an IPBD$((v;w),K)$ exists for all $v \ge v_0$ and $w$ satisfying 
$(\ref{local})$, $(\ref{global})$ and $v > (k-1+\epsilon) w$.
\end{thm}

We note that eliminating $\epsilon$ is not in general possible without knowing more about the set $K$.  That is, for each positive integer $k$ and constant $C>0$, there exists a set of the form $K=\{k,l\}$ such that infinitely many admissible pairs $(v,w)$ satisfying $v \le (k-1)w+C$ fail to admit IPBD$((v;w),K)$.  To see this, it suffices to choose $l-1$ equal to a large prime so that $1=\alpha(\{k,l\}) < \alpha(\{k\})$.  Then, since $v-1$ equals a sum of $w$ terms from $\{k-1,l-1\}$, we are forced to take $v \ge 1+(k-1)w+(l-k)$ whenever $v \not\equiv 1 \pmod{k-1}$.

As a direct consequence of Wilson's theorem and our main result, we can obtain an existence result on PBDs with subdesigns.

\begin{cor}
Let $K \subseteq \Z_{\ge 2}$ with $k=\min K$ and let $\epsilon>0$.  For some $v_0=v_0(K,\epsilon)$,
there exists a PBD$(v,K)$ with a subdesign PBD$(w,K)$ for all $v \ge v_0$ and $w$ satisfying 
$(\ref{local})$, $v(v-1) \equiv w(w-1) \equiv 0 \pmod{\beta(K)}$, and $v > (k-1+\epsilon) w$.
\end{cor}

Techniques in recent papers,\cite{GKLO,Keevash,Keevash2}, provide 
an alternative approach to Theorem~\ref{ipbdk}. However, additional work is required both to 
set up the problem and obtain a bound. 
Our work to follow offers the advantage of an explicit (though technically complicated) construction.  That is, our methods can in
principle directly build specific designs for small $v$ and $w$ as structured combinations of just
a few necessary constituent designs.  The latter are often easy to find in practice through finite-geometric or computer-generated constructions.

In this direction, let us briefly examine some specific sets $K$ to illustrate that a (very nearly) complete existence theory for IPBD$((v;w),K)$ can sometimes be expected without the need for extremely large or random constructions.

An IPBD$((v;w),1+2\Z)$ is known to exist for all odd positive integers $v$ and $w$ satisfying the necessary inequality $v \ge 2w+1$; see \cite{CHL}.  The proof uses one-factorizations of circulant graphs and some specific small designs.  Since there exists a PBD$(v,\{3,5\})$ for all odd positive integers $v$, a folklore result that can be found in, e.g., \cite{CR}, it follows by `breaking up blocks' that there exists an  IPBD$((v;w),\{3,5\})$ under the same conditions.

Similarly, in \cite{CCGL,HH}, it was shown that an IPBD$((v;w),\Z_{\ge 3})$ exists for all integers $v$ and $w$ satisfying $v \ge 2w+3$, with the exception of 7 pairs $(v,w) \in \{(7,2),(8,2),(9,2),(10,2),(11,4),(12,2),(13,2)\}$.  
It is possible again to break up blocks and obtain a result in the case $K=\{3,4,5,6,8\}$.  With perhaps a slight enlargement of the set of exceptions, it is possible to handle the case $K=\{3,4,5\}$.  Some work on the exceptions was done in the undergrad research project of Songfeng Wu at the University of Victoria.

\subsection{Outline}

Our proof of Theorem~\ref{main} is similar in nature to that of
Theorem~\ref{ipbdk} in \cite{DLLholes}.  A key difference is that we work
with group divisible designs having groups of size $\alpha(K)$ in place of
$k-1$.  In our present situation of mixed block sizes, these are more
general objects, not directly corresponding to IPBDs.

The organization of the paper is as follows.
We begin by proving a relaxation in which $w$ is arbitrary but fixed, $v$
is large, and $v-w$ is divisible by a large multiple of $\beta=\beta(K)$.  This is done in Section 2.  
Then, using resolvable and class-uniformly resolvable block designs, we construct two key
families of designs: one in which $v-w$ is arbitrary mod $\beta$ and the other in which $v/w$ approaches $k-1$.  
This is in Section 3.  Next, we give a construction which hits any prescribed
admissible congruence classes for $v$ and $w$.  This primarily number-theoretic work occurs in Section 4.  To complete the
proof of Theorem~\ref{main}, constructions based on transversal designs combine designs with fixed $w$ and large $w$  (near $v/(k-1)$) .  This `analytic' 
side of the argument occurs in Section 5.  We then apply our methods in Section 6 to get existence results for various related
designs, including IPBDs with general index $\lambda$ and mutually orthogonal incomplete latin squares.  A longer argument generalizing our main result to certain group divisible designs is given in Section 7. In both Sections 6 and 7, we include several other applications that 
illustrate the power of using IPBDs as `templates'. These applications include the construction of designs 
with subdesigns and covering and packing problems.

\section{Background on group divisible designs}

Let $T$ denote an integer partition of $v$.
A \emph{group divisible design} of \emph{type} $T$ with block sizes in $K$,
denoted by GDD$(T,K)$ or as a $K$-GDD of type $T$,  is a triple $(V,\Pi,\cB)$ such that
\begin{itemize}
\item
$V$ is a set of $v$ points;
\item
$\Pi=\{V_1,\dots,V_u\}$ is a partition of $V$ into \emph{groups} so that
$T=(|V_1|,\dots,|V_u|)$;
\item
$\cB \subseteq \cup_{k \in K} \binom{V}{k}$ is a set of blocks meeting each
group in at most one point; and
\item
any two points from different groups appear together in exactly one
block.
\end{itemize}

Often in this context, exponential notation such as $n^u$ is used to
abbreviate $u$ parts or `groups' of size $n$.  It is also convenient to drop the 
brackets for a single block size and write $k$ instead of $\{k\}$.
For instance, a
\emph{transversal design} TD$(k,n)$ is a GDD$(n^k,k)$ or a $k$-GDD of type $n^k$. 
In this case, the blocks are transversals of the partition.  A TD$(k,n)$ is equivalent to
$k-2$ mutually orthogonal latin squares of order $n$, where two groups are
reserved to index the rows and columns of the squares.

\begin{thm}[Chowla, Erd\H{o}s and Strauss, \cite{CES}]
\label{asymptotic-td}
Given $k$, there exist TD$(k,n)$ for all sufficiently large integers $n$.
\end{thm}

A group divisible design of type $T=g^u$ is called \emph{uniform}.  Suppose
that, instead of increasing the group size as is done in
Theorem~\ref{asymptotic-td}, we are interested in many groups of a fixed
size.  There is a satisfactory existence result for this situation, which we
state for later use.

\begin{thm}[Draganova, \cite{Draganova}, and Liu, \cite{Liu}]
\label{asym-gdd}
Given $g$ and $K \subseteq \Z_{\ge 2}$, there exist GDD$(g^u,K)$ for all
sufficiently large $u$ satisfying
\begin{eqnarray}
\label{local-gdd}
g(u-1) &\equiv& 0 \pmod{\alpha} ~\text{and}\\
\label{global-gdd}
g^2 u(u-1) &\equiv& 0 \pmod{\beta}.
\end{eqnarray}
\end{thm}

Here, $\alpha=\alpha(K)$ and $\beta=\beta(K)$ are defined as in Section 1.  Let us also define $\gamma:=\beta/\alpha$, which is easily seen to be an integer.  Since $\gcd(\alpha,\gamma)=1$, it follows that (\ref{global-gdd}) can be stated with modulus $\gamma$.

One of the most useful recursive constructions for designs is due to Richard M. Wilson. 
We state it here for group divisible designs. 

\begin{lemma}[Wilson's Fundamental Construction, \cite{ConsUses}]
\label{wfc-lemma}
Suppose there exists a GDD $(V,\Pi,\cB)$, where $\Pi=\{V_1,\dots,V_u\}$.
Let $\omega:V \rightarrow \Z_{\ge 0}$, assigning nonnegative weights to each
point in such a way that for every $B \in \cB$ there exists a
GDD$([\omega(x) : x \in B],K)$.  Then there exists a GDD$(T,K)$, where
$$T=\left[\sum_{x \in V_1} \omega(x),\dots,\sum_{x \in V_u}
\omega(x)\right].$$
\end{lemma}

The constructions to follow use group divisible designs which are uniform, except for one group of a different 
size $w$, i.e. so that $T=g^uw^1$.
Similar to the case of IPBDs, there is a necessary condition taking the form of an 
inequality on $g,u$ and $w$; this is discussed in Section~\ref{guh1}.
Our present focus is on the divisibility conditions.

\begin{prop}
\label{neccond-gdd-h}
The existence of a GDD$(g^uw^1,K)$  implies
\begin{eqnarray}
\label{local-gddw1}
gu & \equiv& 0 \pmod{\alpha}, \\
\label{local-gddw2}
w-g &\equiv &0 \pmod{\alpha}, ~\text{and}\\
\label{global-gddw}
gu(g(u-1) + 2w)  &\equiv& 0 \pmod{\beta}.
\end{eqnarray}
\end{prop}

\begin{proof}
Such a GDD is equivalent to a decomposition of the complete multipartite graph, call it $G$, having exactly $u$ parts of size $g$ and one part of size $w$, into cliques whose sizes belong to $K$.  The degree of a vertex in the part of size $w$ is $gu$, and the degree of other vertices is $g(u-1)+w$.  Each of these quantities is divisible by  $\alpha(K)$, proving (\ref{local-gddw1})
and (\ref{local-gddw2}).  And, since edges of $G$ are partitioned by blocks, $\beta(K)$ divides $2|E(G)| = g^2u(u-1) + 2gwu = gu(gu-g+2w)$; this gives (\ref{global-gddw})
\end{proof}

Of particular importance for us is the case $g=\alpha(K)$, we note 
that the divisibility conditions simplify somewhat in this case.

\begin{prop}
\label{neccond-gdda-h}
The existence of a GDD$(\alpha^uw^1,K)$  implies
\begin{eqnarray*}
\label{localgdda}
w &\equiv &0 \pmod{\alpha}, ~\text{and}\\
\label{globalgdda}
 u(\alpha(u-1) + 2w)  &\equiv& 0 \pmod{\gamma}.
\end{eqnarray*}
\end{prop}


Appealing to Theorem~\ref{asym-gdd}, let $u_{0}=u_{0}(K)$ be a constant such that there exists a 
GDD$(\alpha^u,K)$ for all $u \geq u_{0}$ with $\gamma \mid u$.
Wilson's Fundamental Construction is used to prove the following result.

\begin{lemma} [\cite{DLLholes, DvB}]
\label{sparse1}
Let  $m \geq u_{0}$  with $m\equiv 0\pmod{\gamma}$. There exists a 
constant $s_{0}$ such that for all integers $s \geq s_{0}$ satisfying 
$s\equiv 0\pmod{\alpha}$ and any integer $t$ satisfying 
$0\leq t \leq s$ and $t\equiv 0\pmod{\alpha}$, there exists a 
GDD$(s^mt^1,K)$. 
\end{lemma}

Now let $M_1=\gamma m$ where $m$ is as in Lemma~\ref{sparse1}.

\begin{prop}
\label{sparse-class1}
For any integer $w\equiv 0\pmod{\alpha}$, there exist GDD$(\alpha^uw^1,K)$  for 
all sufficiently large  $u\equiv 0\pmod{M_1}$.
\end{prop}

\begin{proof}
Let $s=\alpha a \gamma$ where $a$ is large enough that there 
exists a GDD$(\alpha^{a\gamma},K)$, by Theorem~\ref{asym-gdd}. We assume 
$s$ is large enough that there exists a GDD$(s^mw^1,K)$ from Lemma~\ref{sparse1}.
Now fill in the groups of size $s$ with GDD$(\alpha^{a\gamma},K)$. This gives us 
a GDD$(\alpha^{a\gamma m}w^1,K)$.
\end{proof}

\section{Resolvable designs}

The constructions in this section are critically important for our existence theory.  The first facilitates covering all necessary congruence classes, and the second
produces constructions close to extremal in the hole-size inequality.

\subsection{Preliminary results}

We say that a design on point set $V$ is \emph{resolvable} if its block
collection $\cB$ can be resolved into partitions of $V$, each of which is
called a \emph{parallel class}.  


We begin by citing a known existence result for resolvable group divisible designs having
fixed group size and a single block size.

\begin{thm}[\cite{CDLL}]\label{asym-rgdd}
Given integers $g \ge 1$ and $k \ge 2$, there exist resolvable GDD$(g^u,\{k\})$
for all sufficiently large integers $u$ satisfying
\begin{eqnarray}
\label{rgdd-neccond1}
gu &\equiv& 0 \pmod{k} ~\text{and} \\
\label{rgdd-neccond2}
g(u-1) &\equiv& 0 \pmod{k-1}.
\end{eqnarray}
\end{thm}

Note that (\ref{rgdd-neccond1}) is necessary for the existence of a parallel class, and (\ref{rgdd-neccond2}) matches (\ref{local-gdd}) for GDDs in the case $K=\{k\}$ with no assumption of resolvability.

In \cite{DLLholes} and \cite{DvB}, resolvable designs are used to 
construct examples of IPBDs in an important congruence class.  The main idea is to consider a `projective' extension in which
each parallel class defines a new point (in the hole) and all of its blocks are extended to include this new point.

Our strategy here, similar to that in \cite{DvB}, is to let $q$ be chosen to admit a GDD$(\alpha^q,K)$.  Then, from a resolvable PBD$(x,\{q-1\})$
we get an IPBD$(1^xy^1,\{q\})$ with $x+y=(q-1)y+1$ from its projective extension.

\begin{prop}
\label{sparse-resol}
Given $K$, a modulus $M_1$, and $w_0=y_0 \alpha$, 
an arbitrary multiple of $\alpha$ (modulo $M_1$), there exists a 
GDD$(\alpha^{x}(\alpha y)^1,K)$ where $x+2y\equiv 1 \pmod{\gamma}$ 
and $y\equiv y_0 \pmod{M_1}$.
\end{prop}

\begin{proof}
Choose a sufficiently large positive integer $q$ such that 
gcd$(q-1,M_1)=1$ and $\gamma | q$ and 
such that there exists an IPBD$(1^xy^1,\{q\})$ where $x+2y\equiv 1 \pmod{q}$.
The necessary conditions for the existence of this IPBD are $x\equiv 0 \pmod{q-1}$, 
$y\equiv 1 \pmod{q-1}$ and $x+2y-1\equiv 0\pmod{q}$. 
Since $\gcd(q-1,M_1)=1$, we can use the Chinese remainder theorem to  
find $y\equiv 1 \pmod{(q-1)}$ and $y\equiv y_0 \pmod{M_1}$ (and a corresponding $x$).
We now apply Wilson's Fundamental Construction to this IPBD$(1^xy^1,\{q\})$ weighting 
each element by $\alpha$. Since $\gamma | q$, there exists a GDD$(\alpha^q,K)$.
The resulting design is a GDD$(\alpha^{x}(\alpha y)^1,K)$ where $x+2y\equiv 1 
\pmod{\gamma}$ and $y\equiv y_0 \pmod{M_1}$.
\end{proof}

Note that in the case $K=\{k\}$, the above construction simplifies somewhat in that we may choose $q=k$ and use PBDs in place of GDDs; see \cite{DLLholes}.
For mixed block sizes, a disadvantage of this construction is that the ratio $x/y$ depends on $q$ in general, and hence may be very large relative to $k$.  The GDD in Proposition~\ref{sparse-resol} is still important for the congruence condition, which is shown later to `generate' all admissible congruence classes.  However, for hole sizes approaching the bound, we aim for parallel classes with mostly blocks of size $k-1$.  

\subsection{Class-uniformly resolvable designs}
\label{curds}

A design is \emph{class-uniformly resolvable} if it is resolvable in such a way that each parallel class has the same multiset of block sizes.  Such a multiset is typically specified in advance, as in a scheduling problem.  Class-uniformly resolvable designs were introduced in \cite{LRVcurds}, where they were studied for the first natural set of block sizes $K=\{2,3\}$.  Abbreviations such as `CURD' or `CURGDD' have taken hold, and explicit constructions have been found in several cases.

For technical reasons, our main construction requires an ingredient design with a large hole which is arbitrarily close to the upper bound dictated by $k = \min K$ but which uses a mixture of block sizes in $K$.  A CURD with block sizes in $\{k-1: k \in K\}$ is well-suited to this.  To build such a CURD, we employ the following theorem on resolvable graph decompositions  of complete multipartite graphs.

\begin{thm}[\cite{CDLL,DL}]
\label{grgdd}
Let $G$ be a graph with $n$ vertices, $m>0$ edges, and vertex degrees $d_1,\dots,d_n$.  Define $\alpha^*=\alpha^*(G)$ to be the least positive integer $a$ such that the vector $a(1,n/2m)$ is an integral linear combination of $(d_i,1)$, $i=1,\dots,n$. Under the assumption that $\gcd(n,\alpha^*)=1$, there exists a resolvable $G$-GDD of type $g^u$ for any positive integer $g$ and all integers $u \ge u_0(G,g)$ satisfying $gu \equiv 0 \pmod{n}$ and $g(u-1) \equiv 0 \pmod{\alpha^*}$.
\end{thm}

To apply this result to CURDs, consider graphs built from a vertex-disjoint union of cliques.  If we can obtain a certain resolvable decomposition into such a graph, it can be regarded as a CURD whose class types have a pre-specified proportion of block sizes.

We now consider the details necessary for our particular use of Theorem~\ref{grgdd}.  First, we may replace $K$ by a finite subset $\{k_1,k_2,\dots,k_t\}$ of $K$ having the same parameters $\alpha$ and $\beta$, and we may assume $k_1<k_2<\dots<k_t$ for some integer $t\ge 2$.  Fix a real number $\theta>0$ and build a graph $G_{K,\theta}$ as a vertex-disjoint union of $n_i$ cliques of size $k_i-1$, $i=1,\dots,t$.  We would like to choose $n_1 = O(1/\theta)$ and $n_i$ bounded for $i=2,\dots,t$, so that the graph approximates a disjoint union of cliques of the smallest size, yet has the desired $\alpha^*$.

To this end, put $n := |V(G_{K,\theta})|$ and $m:=|E(G_{K,\theta})|$.  
Recall that
\begin{equation}
\label{n-over-alpha}
n= \sum_{i=1}^t n_i (k_i-1),
\end{equation}
and that $\gcd(\{(k_i-1)/\alpha:i=1,\dots,t\}) = 1$.
By Schur's theorem on conical combinations of integers (the `coin problem'), there exists a choice of nonnegative integers $n_2,\dots,n_t$ so that $n':=\sum_{i=2}^t n_i(k_i-1)/\alpha$ is relatively prime to $(k_1-1)/\alpha$.  By Dirichlet's theorem on primes in arithmetic progressions, there exists a choice of $n_1 = O(1/\theta)$ so that $n_1(k_1-1)/\alpha+n'$ equals a large prime, say $p$.  From (\ref{n-over-alpha}), we have $n=p\alpha$ vertices in our graph.

Considering the sum of vertex degrees, we have
$$\frac{2m}{\alpha} = \sum_{i=1}^t \frac{n_i(k_i-1)}{\alpha} (k_i-2) \equiv \sum_{i=1}^t \frac{n_i(k_i-1)}{\alpha} (-1) \equiv -\frac{n}{\alpha} = -p \pmod{\alpha}.$$ 
Recall that the set of integers $a$ such that the vector $a(1,n/2m)$ is an integral linear combination of $(d_i,1)$, $i=1,\dots,n$ forms the ideal $\langle \alpha^* \rangle \subseteq \Z$.
It follows that $\alpha^*(G_{K,\theta})$ divides $2m/\alpha$ and is coprime to $\alpha$.  Also, with $n_1$ sufficiently large (but depending only on $K$), we can ensure that $2m$ is very close to $n(k_1-2)$, say
$$p \alpha(k_1-2) < 2m < p (\alpha(k_1-2)+1).$$ 
With this choice, $2m$ is coprime to $p$ and it follows that $\gcd(n,\alpha^*)=1$ for $G_{K,\theta}$.    Theorem~\ref{grgdd} can then be invoked on $G_{K,\theta}$ and with group size $\alpha$.  The necessary and asymptotically sufficient conditions for the existence of the class-uniformly resolvable GDD based on $G_{K,\theta}$ are
$u \equiv 0 \pmod{p}$ and $u \equiv 1 \pmod{\alpha^*}$.  By choosing $n_1$ large, we may demand that the number of parallel classes is only a proportion $\theta$ less than that in a hypothetical resolvable GDD with block size $k_1-1$.  That is, the number of parallel classes in our class-uniformly resolvable GDD is at least $\alpha(u-1)(1-\theta)/(k_1-2)$.

\section{Arbitrary congruence classes}

An \emph{incomplete group divisible design}, or IGDD, is a quadruple
$(V,\Pi,\Xi,\cB)$ such that $V$ is a set of $v$ points, $\Pi =
\{V_1,\dots,V_u\}$ is a partition of $V$ into `groups',
$\Xi=\{W_1,\dots,W_u\}$ with $W_i \subseteq V_i$ called `holes' for each
$i$, and $\cB \subseteq \cup_{k \in K} \binom{V}{k}$ is a set of blocks such
that
\begin{itemize}
\item
two points get covered by a block (exactly one block) if and only if they
come from different groups, say $V_i$ and $V_j$, $i \neq j$, {\bf and} they
do not both belong to the corresponding holes $W_i$ and $W_j$.
\end{itemize}

Similar to GDDs, the type of an IGDD is a list of the pairs $(|V_i|;|W_i|)$ of
group/hole sizes, and the set of block sizes is typically indicated.  We are
interested here in the `uniform' case IGDD$((g;h)^u,K)$.  We state the
`divisibility' conditions for such designs.

\begin{prop}
\label{neccond-igdd}
The existence of an IGDD$((g;h)^u,K)$ implies
\begin{eqnarray}
\label{local-igdd}
g(u-1)~\equiv~h(u-1) &\equiv& 0 \pmod{\alpha(K)},  ~\text{and} \\
\label{global-igdd}
(g^2-h^2)u(u-1) &\equiv& 0 \pmod{\beta(K)}.
\end{eqnarray}
\end{prop}

We say integers $g,\ h$, and $u$ are \emph{admissible} if (\ref{local-igdd})
and (\ref{global-igdd}) hold.

Now, let $k=\min K$.  A similar counting argument as for
Proposition~\ref{holesize-bound} gives
$$\frac{hu(g-h)(u-1)}{k-1} \binom{k-1}{2} \le (g-h)^2 \binom{u}{2},$$
or
\begin{equation}
\label{ineq-igdd}
g \ge (k-1)h.
\end{equation}

From the theory of `edge-colored graph decompositions' (see \cite{LW}), we
have an asymptotic existence result (in $u$) for uniform IGDDs.  The proof
is sketched for $K=\{k\}$ in \cite{DLLholes} and given in full detail for
general $K$ in the thesis \cite{CvB}. The basic idea is to apply the main
result of \cite{LW} using $g^2-h^2$ edge colors.  In doing so, inequality
(\ref{ineq-igdd}) is needed for a nonnegative rational decomposition using a
certain edge-colored graph family.

\begin{thm}
\label{igdd-exist}
Given integers $g,h,k$ with $k \ge 2$ and $g \ge (k-1)h$, there exists an
IGDD$((g;h)^u,K)$ whenever $u$ is sufficiently large satisfying {\rm
(\ref{local-igdd})} and {\rm (\ref{global-igdd})}.
\end{thm}

Groups of an IGDD can be filled with IPBDs, and excess hole points can be
identified.  This is a standard filling construction, stated here without
proof for later use.  The case of a single block size appears in
\cite{DLLholes}.

\begin{cons}
\label{igdd-fill}
Suppose there exists an IGDD$((g;h)^u,K)$ and an IPBD$((x;y),K)$ with
$g-h=x-y$ and $y \ge h$.  Then there exists an IPBD$((v;w),K)$ with
$v-w=u(x-y)$ and $w=(u-1)h+y$.
\end{cons}

We can also fill in the groups of an IGDD with GDDs.

\begin{cons}
\label{igdd-fillgdd}
Suppose there exists an IGDD$((g;h)^u,K)$ and a GDD$(\alpha^x(\alpha y)^1,K)$
where $g-h =\alpha x$ and $h \le \alpha y$. Then there exists a GDD$(\alpha^nw^1)$ 
where $n=ux$ and $w=h(u-1) + \alpha y\equiv 0\pmod{\alpha}$.
\end{cons}

We use the designs constructed in Propositions~\ref{sparse-resol} together with 
this construction to produce the remaining examples for admissible congruence 
classes. We construct GDD$(\alpha^nw^1,K)$ where $w=\alpha s$ for 
admissible congruence classes of $n$ and $w$ modulo $M_1$.

\begin{prop}
\label{fixed-arb}
Given K, a positive modulus $M_1$, and admissible congruence classes $n_0$ and $s_0$ 
modulo $M_1$ for GDDs of type $\alpha^{n_0} (\alpha s_0)^1$ with block sizes $K$, 
there exist GDD$(\alpha^n(\alpha s)^1, K)$ for infinitely many  $n\equiv n_0 \pmod{M_1}$ 
and $s\equiv s_0\pmod{M_1}$.
\end{prop}

\begin{proof}
We  use Construction~\ref{igdd-fillgdd}  and IGGD$((g;h)^u,K)$ and fill in the groups 
with GDD$(\alpha^xy^1,K)$ from Prop ~\ref{sparse-resol} to construct a GDD$(\alpha^{ux}w^1,K)$ where 
$w=h(u-1) +\alpha y$ and $g-h=\alpha x$. Let $n=ux$ and $w=\alpha s$.

We get the following equations involving $u,h,n,s$.  First, 
$v=\alpha n+ w$ and $v-w = \alpha n =\alpha ux$. 
Substituting $x\equiv 1-2y \pmod{\gamma}$, this 
becomes
$\alpha ux \equiv \alpha u (1-2y) \pmod{\gamma}$.

Next use $\alpha y = w-h(u-1)$ to get 
$$\alpha n = \alpha ux \equiv \alpha u -2u (w-h(u-1)) \pmod{\gamma}.$$ 
Rewriting again using $w=\alpha s$, 
we get the following congruence on $n=ux,s,u,h$:
\begin{equation}
\label{u-cong}
\alpha n - \alpha u(1-2s) \equiv 2u(u-1)h \pmod{\gamma}.
\end{equation}
We want to find $u$ and $h$ such that $n\equiv n_0\pmod{M_1}$ and 
$s\equiv s_0 \pmod{M_1}$. (Note that $\gamma \mid M_1$.) 
So we consider (\ref{u-cong}) separately modulo each prime power $p^t$ such that $p^t \parallel M_1$. 
To solve (\ref{u-cong}), let us choose
\begin{equation*}
u \equiv \begin{cases}
\frac{1}{2} & \text{modulo odd prime powers}  \\
n_0(1-2s_0)^{-1}  & \text{modulo powers of 2}.
\end{cases}
\end{equation*}
In the first case for $u=\frac{1}{2}$, the congruence on $h$ becomes 
 $h\equiv \alpha(1-2s_0-2n_0)$.
For the case of powers of 2, the congruence for $h$ is even simpler: $h\equiv 0$.
Note that in this case, this gives us $y\equiv s_0$ and $x\equiv 1-2s_0$, 
and this means that $1-2s_0$ divides $n_0$ and $u$ is an integer.
Now we use the Chinese remainder theorem to provide  a simultaneous 
solution for $u$ and $h$ modulo $M_1$. 
We summarize the choice of parameters in Table~\ref{params2}.

\begin{table}[htbp]
$$\begin{array}{l|c|c}
&   \text{odd prime powers} & \text{powers of 2}    \\
\hline
h \equiv & \alpha(1-2s_0-2n_0)  & 0 \\
u \equiv & \frac{1}{2} & n_0(1-2s_0) ^{-1}
\end{array}$$
\begin{eqnarray*}
\alpha y &\equiv& \alpha s_0-(u-1)h \\
x &\equiv&  1-2y \\
g &\equiv&  \alpha x + h.
\end{eqnarray*}
\caption{Choice of parameters for congruence classes modulo primes dividing $M_1$.}
\label{params2}
\end{table}

Notice that we compute a large integer $y \pmod{M_1}$ from $u$ and $h$ 
and then use $y$ to find $x$ and $g$. 
This set-up allows us to  make $n$ and $w$ (or $s$) arbitrarily large 
by increasing the choice of $y$. 

We now check that the necessary conditions for the existence of 
IGGD$((g;h)^u,K)$ are satisfied. 
Since $w\equiv 0\pmod{\alpha}$ and $w=\alpha s =  \alpha y + h(u-1)$, 
we have $(u-1)h \equiv 0\pmod{\alpha}$. 
Since $g(u-1) = \alpha x(u-1) +h(u-1)$, 
we also have $(u-1)g \equiv 0\pmod{\alpha}$. 
So (\ref{local-igdd}) holds. 
To verify (\ref{global-igdd}),  we first compute 
\begin{eqnarray*}
(g+h)(u-1) & = & (u-1)(g-h) + 2h(u-1) \\
                 & =  & (u-1)(\alpha x) + 2\alpha s_0 - 2 \alpha y \\
                 & = & \alpha ux - \alpha x + 2\alpha s_0 - 2\alpha y \\
                 & \equiv & \alpha n_0 + 2\alpha s_0 - \alpha (x+2y) \\
                 & \equiv & \alpha n_0  + 2\alpha s_0 - \alpha \pmod{\gamma} \\
                 & \equiv & \alpha  (n_0-1) + 2w_0 \pmod{\gamma}.
\end{eqnarray*}
This gives us 
\begin{eqnarray*}
(g^2-h^2)u(u-1) ) & = &  (g-h) u (g+h)(u-1) \\
   & \equiv & (\alpha n_0) (\alpha (n_0-1) + 2w_0) \pmod{\gamma} \\
   & \equiv & 0 \pmod{\beta}.
\end{eqnarray*}

To see the last step, we note that $ (\alpha (n_0-1) + 2w_0) \equiv 0 \pmod{\gamma}$ 
is necessary condition (\ref{global-gdd}) for the existence of a GDD$(\alpha^{n_0}{w_0}^1,K)$. 
Thus, for sufficiently large $u$, the IGDD$((g;h)^u,K)$ required for 
Construction~\ref{igdd-fillgdd} exists and we can construct 
GDD$(\alpha^n(\alpha s)^1,K)$ with $n\equiv n_0 \pmod{M_1}$ and 
$s\equiv s_0 \pmod{M_1}$. 
\end{proof}

We are now in a position to prove an asymptotic result for GDD$(\alpha^nw^1,K)$ for 
fixed $w$.

\begin{thm}
\label{fixed-hole}
Given $w\equiv 0\pmod{\alpha(K)}$, there exist GDD$(\alpha^nw^1,K)$ for all 
sufficiently large $n$ satisfying $\gamma \mid n(\alpha(n-1) + 2w)$.
\end{thm}

\begin{proof}

Let $w=\alpha s$ and take a large $n$ as above. We would like 
to construct a GDD$(\alpha^nw^1,K)$. Let $n=n_1 +n_2 +n_3$ where 
$n_1\equiv 0\pmod{M_1}$, $n_2\equiv n_0 \pmod{M_1}$, and $n_3\equiv 0\pmod{M_1}$,
and the $n_i$ are all sufficiently large for $i=1,2,3$.
We first use Proposition~\ref{sparse-class1} to construct a GDD$(\alpha^{n_1}w_1^1,K)$ 
where $w_1 =\alpha n_2 + \alpha s_2$ and $s_2\equiv s \pmod{M_1}$. Next we 
use Proposition~\ref{fixed-arb} to construct a GDD$(\alpha^{n_2}(\alpha s_2)^1,K)$ 
where $n_2\equiv n_0 \pmod{M_1}$ and $s_2\equiv s\pmod{M_1}$ and replace the 
fixed group of size $w_1$ with this GDD. This gives us a
GDD$(\alpha^{n_1 + n_2}(\alpha s_2)^1,K)$. Now let $\alpha s_2 =\alpha n_3 + \alpha s$, 
and once again use Proposition~\ref{sparse-class1}. The resulting design is a 
GDD$(\alpha^{n_1+n_2+n_3}(\alpha s)^1,K)$ or a GDD$(\alpha^nw^1,K)$.
\end{proof}

\section{Recursion}

To emulate the recursive construction strategy in \cite{DLLholes}, we require six ingredient designs.  In the present setting, these are `consecutive' GDDs of type $\alpha^t h^1$ and $\alpha^{t+1} h^1$ for each of three different sizes for $h$: two nearby fixed values and one large value (that depends on $t$). The case of large $h$ is treated next.  

\begin{lemma}
\label{xstar}
Let $K \subseteq \Z_{\ge 2}$ with $k=\min K$, and let $\epsilon>0$.
There exist $K$-GDDs of type $\alpha^t (\alpha x_\star)^1$ and $\alpha^{t+1} (\alpha x_\star)^1$ for some integers $t \equiv -1 \pmod{\gamma}$ and $x_\star$ satisfying
$t/x_\star  < k-2+\epsilon$.
\end{lemma}

\begin{proof}
We begin by setting up some ingredient designs.  First, extend parallel classes of a resolvable GDD$(\alpha^{t_1},\{k-1\})$ to produce a GDD$(\alpha^{t_1} (\alpha y_\sharp)^1,\{k\})$, where  $y_\sharp = (t_1-1)/(k-2)$.  
Next,  extend a class-uniformly resolvable GDD based on a resolvable $G_{K,\theta}$-GDD of type $\alpha^{t_2}$; see Section~\ref{curds}.  This yields a GDD$(\alpha^{t_2} (\alpha y_\star)^1,K)$, where $y_\star = (1-\theta)(t_2-1)/(k-2)$.  Here, $\theta$ is a parameter chosen sufficiently small and in terms of $\epsilon$.

We examine properties of $t_1$ and $t_2$.  First, recall the  necessary conditions $k-1 \mid \alpha t_1$ and $k-2 \mid \alpha (t_1-1)$ for existence of the required resolvable GDD.  
The congruence conditions on $t_2$ for existence of the required  class-uniformly resolvable GDD are $\alpha^*(G_{K,\theta}) \mid t_2-1$ and $p \mid t_2$, where $p:=|V(G_{K,\theta})|$ is, say, chosen to be a large prime.
For technical reasons to follow, it is convenient to choose $t_2$ coprime with $\gamma t_1$.  This can be accomplished with a selection $t_2 \equiv 1 \pmod{\mathrm{lcm}(\gamma,\frac{k-1}{\alpha},\alpha^*(G_{K,\theta}))}$, noting that $\frac{k-1}{\alpha} \mid t_1$ from above.  Since $p$ is coprime with $\alpha^*$ and can be chosen larger than the other modulus, we have from the Chinese Remainder Theorem (an increasing sequence of) solutions to these simultaneous congruences on $t_2$.
Finally, let us additionally choose both $t_i$ sufficiently large so that, by Theorem~\ref{fixed-hole}, there exist GDD$(\alpha^{t_i}(\alpha y_i)^1,K)$ and GDD$(\alpha^{t_i}(\alpha y_i+\beta)^1,K)$, $i=1,2$, where $y_1$ and $y_2$ are fixed small hole sizes to be determined later.

Since $\gcd(\gamma t_1,t_2)=1$, it follows that we may choose integers $n_1$ and $n_2$ such that $n_1t_1 - n_2 t_2 = 1$ and $\gamma \mid n_1$.    We can ensure that the $n_i$ are sufficiently large for the existence of transversal designs TD$(t_i+1,n_i)$ and also GDD$(\alpha^{n_i} (\alpha z_i)^1,K)$ for fixed integers $z_1$ and $z_2$ to be determined later.

Apply Wilson's Fundamental Construction to the TD$(t_i+1,n_i)$ using weight $\alpha$ on all points in the first $t_i$ groups.  On the last group of the TD with $i=1$, use weights $\alpha y_1,\alpha y_1+\beta$ and $\alpha y_\sharp$.  On the last group of the TD with $i=2$, use weights $\alpha y_2,\alpha y_2+\beta$ and $\alpha y_\star$.  Replace blocks of the TD with the above ingredient GDDs, as appropriate.  We additionally fill the resulting first $t_i$ groups with GDD$(\alpha^{n_i} (\alpha z_i)^1,K)$ using $\alpha z_i$ new points.  These join the last group to act as a hole.  With $t:=n_2t_2$, the resulting GDDs have types $\alpha^{t} (\alpha x_\star)^1$ and $\alpha^{t+1} (\alpha x_*)^1$, where each hole size is a combination of the respective weights.  It remains to show that the weights can be chosen so that the hole sizes align ($x_\star=x_*$), while additionally being relatively large ($(t-1)/x_\star < k-2+\epsilon$).

In what follows, we describe the selection of weights with integer parameters $\mu_i,\nu_i$ (to be determined later) satisfying $0 \le \mu_i \le \nu_i \le n_i$.  We use $\mu_i$ weights equal to $\alpha y_i+\beta$, $\nu_i-\mu_i$ weights equal to $\alpha y_i$, and $n_i-\nu_i$ large weights ($\alpha y_\sharp$ and $\alpha y_\star$ for $i=1,2$, respectively).  The overall hole sizes are, after weighting and filling, $\alpha x_* = (n_1-\nu_1) \alpha y_\sharp + \nu_1 \alpha y_1 + \mu_1 \beta + \alpha z_1$ and 
$\alpha x_\star = (n_2-\nu_2) \alpha y_\star + \nu_2 \alpha y_2 + \mu_2 \beta + \alpha z_2$.

Equating hole sizes, the condition $x_*=x_\star$ amounts to 
$$(n_1-\nu_1) y_\sharp + \nu_1 y_1 + \mu_1 \gamma + z_1 = (n_2-\nu_2) y_\star + \nu_2 y_2 + \mu_2 \gamma + z_2,$$
or, rearranging,
\begin{equation}
\label{equal-holes}
\nu_1 (y_\sharp-y_1) - \nu_2(y_\star-y_2) =  n_1 y_\sharp - n_2 y_\star + (\mu_1-\mu_2) \gamma + z_1 -z_2.
\end{equation}
We show that a selection of weights achieving (\ref{equal-holes}) is possible by approximately equating the left side to the first terms on the right, and then using $\mu_i$ and $z_i$ for `fine tuning'.
To this end, let
$$\nu_1 = \left\lceil \frac{n_2(1+\theta(t_2-1))}{(k-2)(y_\sharp-y_1)} \right\rceil~~\text{and}~~
\nu_2  =\left\lceil \frac{n_1}{(k-2)(y_\star-y_2)} \right\rceil.$$
We estimate $\frac{\nu_1}{n_1} \approx \frac{n_2}{n_1} (\frac{1+\theta t_2}{t_1}) \approx \frac{1+\theta t_2}{t_2} = \theta + \frac{1}{t_2}$.  It follows that this ratio is arbitrarily small as a function of $\theta$ (since in particular $t_2 \gg 1/\theta$).  A similar calculation gives 
$\frac{\nu_2}{n_2} \approx \frac{t_2}{t_1 (1-\theta)(t_2-)} \approx \frac{1}{t_1}.$ 
Since $\theta$ is small and $n_1/t_2 \approx n_2/t_1$, we have $\nu_i \ll n_i$ for each $i$. 

The dominant term in the right side of (\ref{equal-holes}) is, say, 
$$D:=n_1 y_\sharp - n_2 y_\star = \frac{1}{k-2} [1-n_1+n_2+\theta n_2(t_2-1) ],$$
where in the last step we use $n_1t_1-n_2t_2=1$.
After a calculation using the definition of $\nu_i$,
\begin{equation}
\label{hole-diff-est}
|\nu_1 (y_\sharp-y_1) - \nu_2 (y_\star-y_2) - D| \le \max(y_\sharp,y_\star),
\end{equation}
a quantity independent of the $n_i$.

The fine-tuning is accomplished, then, by selecting nonnegative integer weights $\mu_i$ 
for (\ref{equal-holes}) 
so that $(\mu_1-\mu_2)\gamma$ is the smallest multiple of $\gamma$ not less than
$\nu_1 (y_\sharp-y_1) - \nu_2 (y_\star-y_2) - D$.  For this purpose, one of $\mu_1$ or $\mu_2$ could equal zero, according to the required sign.
The key point is that, by our estimate (\ref{hole-diff-est}) 
and our choice of $n_i$ sufficiently large, we can ensure $\mu_i \leq \nu_i$ for both $i=1,2$.

Now, by our assumption that $\gamma \mid n_1$, we have the existence of GDD$(\alpha^{n_1} (\alpha z_1)^1,K)$ with no congruence restriction on $z_1$ whatsoever.  
It follows that there is a choice of $z_1$, an integer in $\{0,1,\dots,\gamma-1\}$, so that (\ref{equal-holes}) holds modulo $\gamma$.
With only finitely many possibilities ahead of time, we could take $n_1$ sufficiently large so that this GDD exists for any choice of $z_1$.  

If in the above construction, we select the $t_i$ sufficiently large and $0 < \theta \ll \zeta \ll \epsilon$, we can ensure that $\frac{\nu_i}{n_i} < \zeta$.  Then we have
\begin{equation*}
\frac{t}{x_\star} = \frac{t_2n_2}{(n_2-\nu_2) y_\star + \nu_2 y_2 + \mu_2 \gamma + z_2}
< \frac{t_2(k-2)}{(1-\zeta)(1-\theta)(t_2-1)}
< k-2+\epsilon,
\end{equation*}
as desired.
\end{proof}

\rk Observe that the number of groups $t$ can be forced arbitrarily large by choosing $\epsilon$ sufficiently small.

Now, following \cite[\S 5]{DLLholes}, we obtain a versatile family of GDDs by combining the preceding ingredient designs.

\begin{lemma}
\label{ABC}
Let $K \subseteq \Z_{\ge 2}$ with $k=\min K$, and let $\epsilon>0$.  For some integer $t$, there exists a $K$-GDD of type $(\alpha A)^t (\alpha B)^1 (\alpha C)^1$ for all sufficiently large integers $A,B,C$ satisfying $\gamma \mid A,C$ and $B \le A \le C \le At/(k-2+\epsilon)$.
\end{lemma}

\begin{proof}
Using Lemma~\ref{xstar}, let $t \equiv -1 \pmod{\gamma}$ and $x_*$ be chosen so that there exist $K$-GDDs of each of the following types:
$$
\begin{array}{lll}
\alpha^{t+1},& \alpha^t (\alpha+\beta)^1,& \alpha^t (\alpha x_*)^1,\\
\alpha^{t+2},& \alpha^{t+1} (\alpha+\beta)^1, &\alpha^{t+1} (\alpha x_*)^1,
\end{array}$$
where additionally $t/x_* < k-2+\epsilon/2$.
 
Take a TD$(t+2,A)$ for large $A$, where $\gamma \mid A$, and truncate $A-B$ points from the second last group.  Give weight $\alpha$ to all points not in the last group, and weights in $\{\alpha, \alpha+\beta, \alpha x_*\}$ to points of the last group.  Use Wilson's Fundamental Construction, replacing blocks of the truncated TD, whose sizes are in $\{t+1,t+2\}$, with $K$-GDDs of the types given above.
The result is a $K$-GDD of type $(\alpha A)^{t} (\alpha B)^1 (\alpha C)^1$, where $C$ is any sum of $A$ terms from the set $\{1,1+\gamma,x_*\}$.
 
It remains to analyze the possible values of $C$.  
For $0 \le i \le A$, put $$\Gamma_i := \{i+(A-i)x_*,i+(A-i)x_*+\gamma,\dots,i+(A-i)x_*+\gamma i\},$$ the arithmetic progression of possible sums $C$ in which exactly $A-i$ of the summands equal $x_*$.
 By comparing endpoints, $\Gamma_i \cap \Gamma_{i-1} \neq \emptyset$ for indices $i$ in the range $(x_*-1)/\gamma \le i \le A$.  It follows that the realizable values of $C$ cover the arithmetic progression $\{A,A+\gamma,A+2\gamma,\dots,D\}$, where 
\begin{equation}
\label{D-est}
D > (A-(x_*-1)/\gamma) x_* = \frac{At}{k-2+\epsilon/2} \left[ 1- \frac{x_*-1}{\gamma A} \right].
\end{equation}
For sufficiently large $A \gg t$, the right side of (\ref{D-est}) can be made larger than $At/(k-2+\epsilon)$.
\end{proof}

We are now in a position to prove our main result on incomplete pairwise balanced designs.

\begin{proof}[Proof of Theorem~\ref{main}]
Suppose we are given large integers $v$ and $w$ satisfying the divisibility conditions (\ref{local}) and (\ref{global}) for IPBDs, and additionally satisfying $v > (k-1+\epsilon)w$.  Write $v-w=\alpha (tA+B)$ and $w=\alpha (C + z) + 1$, where $\gamma \mid A,C$, $z \in \{0,1,\dots,\gamma-1\}$, and $t$ is chosen as in Lemma~\ref{ABC} with $\epsilon/2$ taking the role of $\epsilon$.
We may assume $A,B,C$ are chosen sufficiently large so that, by Theorem~\ref{fixed}, there exist both
IPBD$((\alpha (A+z)+1;\alpha z+1),K)$ and IPBD$((\alpha (B+z)+1;\alpha z+1),K)$ for each of the $\gamma$ possible values of $z$.
Take the $K$-GDD of Lemma~\ref{ABC}, add $\alpha z+1$ new points, and fill all but the last group with the above IPBDs. 
The result is an IPBD$((v;w),K)$, where observe that the added points join the last group of our GDD to become the hole.  

There are finitely many values of $w$ that our construction does not cover; for each, we invoke Theorem~\ref{fixed} to get existence of IPBD$((v;w),K)$ for sufficiently large $v$.
\end{proof}

\section{Applications}
\label{ipbd-applications}

\subsection{Arbitrary index $\lam$}

We consider here an extension of IPBDs which allows multiply-covered pairs of points.
Let $v,w,K$ be as before, and let $\lambda$ be a nonnegative integer.  An \emph{incomplete pairwise balanced design of index} $\lam$, denoted IPBD$_\lam((v;w),K)$, is a triple
$(V,W,\cB)$ where
\begin{itemize}
\item
$V$ is a set of $v$ points and $W \subset V$ is a \emph{hole} of size $w$;
\item
$\cB \subseteq \cup_{k \in K} \binom{V}{k}$ is a family of blocks;
\item
no two distinct points of $W$ appear in a block; and
\item
every two distinct points not both in $W$ appear together in exactly $\lam$
blocks.
\end{itemize}

The necessary divisibility conditions weaken accordingly.

\begin{prop}
\label{neccond-lam}
The existence of an IPBD$_\lam((v;w),K)$ implies
\begin{eqnarray}
\label{local-lam}
\lam(v-1)~\equiv~\lam(w-1) &\equiv& 0 \pmod{\alpha(K)}, ~\text{and}\\
\label{global-lam}
\lam v(v-1) - \lam w(w-1) &\equiv& 0 \pmod{\beta(K)}.
\end{eqnarray}
\end{prop}

In the case $w=0$ or $w=1$, all pairs of points get covered exactly $\lam$ times, and the notation PBD$_\lam(v,K)$ may be used.  In \cite{RMW2}, R.M.~Wilson proved existence of PBD$_\lam(v,K)$ for fixed $K,\lam$ and sufficiently large $v$ satisfying the divisibility conditions.  Using this and our main result, we have the following existence result for IPBDs of general index 
$\lam$.

\begin{thm}
\label{main-lam}
Let $\lam \in \Z_{\ge 0}$, $K \subseteq \Z_{\ge 2}$ with $k=\min K$, and $\epsilon>0$.  For some $v_0=v_0(K,\lam,\epsilon)$,
an IPBD$_\lam((v;w),K)$ exists for all $v \ge v_0$ and $w$ satisfying 
$(\ref{local-lam})$, $(\ref{global-lam})$ and $v > (k-1+\epsilon) w$.
\end{thm}

\begin{proof}
Let $L=\{u\ge 2 : \exists \, \text{PBD}_\lam(u,K)\}$.  Observe that $\min L = k$, so that by Theorem~\ref{main} there exists IPBD$((v;w),L)$ for all $v \ge v_0(L,\epsilon)$ and any $w$ satisfying the given inequality, where
$v \equiv w \equiv 1 \pmod{\alpha(L)}$ and $v(v-1)-w(w-1) \equiv 0 \pmod{\beta(L)}$. 

Put $v_0(K,\lam,\epsilon):=v_0(L,\epsilon)$ and suppose $v,w$ are given satisfying $v \ge v_0$, $v > (k-1+\epsilon) w$, (\ref{local-lam}) and (\ref{global-lam}).  
The latter two divisibility conditions can be rewritten
\begin{eqnarray*}
v-1~\equiv~w-1 &\equiv& 0 \pmod{\frac{\alpha(K)}{\gcd(\lam,\alpha(K))}}, ~\text{and}\\
v(v-1) - w(w-1) &\equiv& 0 \pmod{\frac{\beta(K)}{\gcd(\lam,\beta(K))}}.
\end{eqnarray*}
Following a similar argument for PBD$_\lam(v,K)$, in \cite[Proposition 9.2]{RMW1} it was shown that $\alpha(L)=\alpha(K)/\gcd(\lam,\alpha(K))$ and $\beta(L)=\beta(K)/\gcd(\lam,\beta(K))$, or twice this number if it is odd.  Therefore, there exists an IPBD$((v;w),L)$, say $(V,W,\mathcal{B})$. 
Replace each block $B \in \mathcal{B}$, say with size $|B|=u \in L$, by the block set of a PBD$_\lam(u,K)$ on the points of $B$.  The result is an IPBD$_\lam((v;w),K)$.
\end{proof}

\subsection{Incomplete latin squares}
A \emph{latin square} is an $n \times n$ array with entries from an
$n$-element set of symbols such that every row and column is a permutation
of the symbols.  (The symbol set $[n]:=\{1,\dots,n\}$ is conveneint, since it 
matches the row and colum indices.)

Two latin squares of side $n$ are \emph{orthogonal} if, when superimposed, all ordered pairs of symbols
occur exactly once among the $n^2$ cells.  A set of latin squares in which any pair are orthogonal is 
a set of \emph{mutually orthogonal latin squares}, or `MOLS' for
short.  The maximum size of a set of MOLS of order $n$ is denoted $N(n)$.
By a straightforward argument, $N(n) \le n-1$ for $n>1$,  with equality if and only
if there exists a projective plane of order $n$.  Consequently, $N(q)=q-1$
for prime powers $q$.  Using number sieves and some special constructions (adopted from earlier work),
Beth showed in \cite{Beth} that $N(n) \ge n^{1/14.8}$ for all sufficiently large $n$.

An \emph{incomplete latin square} of side $n$ with a \emph{hole} of
size $m$ is an $n \times n$ array $L=(L_{ij}: i,j \in [n])$ on $n$ symbols
(let us say $[n]$ once again) together with a \emph{hole} $M \subseteq
[n]$ such that
\begin{itemize}
\item
$L_{ij}$ is empty if $\{i,j\} \subseteq M$;
\item
$L_{ij}$ contains exactly one symbol if $\{i,j\} \not\subseteq M$;
\item
every row and every column in $L$ contains each symbol at most once; and
\item
symbols in $M$ do not appear in rows or columns indexed by $M$.
\end{itemize}
Incomplete latin squares are interesting in that they furnish a `border' for smaller squares to embed in larger ones.
An example in the case $n=5$, $m=2$ is shown below.
\begin{center}
\begin{tabular}{|ccccc|}
\hline
 &  & 3 & 4 & 5 \\
 &  & 4 & 5 & 3 \\
3 & 4 & 1 & 2 & 5 \\
4 & 5 & 2 & 3 & 1 \\
5 & 3 & 4 & 1 & 2 \\
\hline
\end{tabular}
\end{center}

Orthogonality can be extended to incomplete latin squares.
Two incomplete latin squares $L,L'$ with common hole on symbols in $M$ 
are \emph{orthogonal} if, when superimposed, all ordered pairs not in $M\times M$
occur exactly once among the (common) nonblank cells.
The `mutually orthogonal' terminology for sets of latin squares also applies to 
sets of incomplete latin squares with a common side and hole.
Let us abbreviate a set of $t$ mutually orthogonal incomplete squares of side $n$ with holes of
size $m$ by $t$-IMOLS$(n;m)$. The case $m=0$ or $1$
reduces to ordinary MOLS. 
It is a straightforward counting argument that the
existence of $t$-IMOLS$(n;m)$ requires
$n \ge (t+1)m$.  The case $t=1$ is the familiar condition that latin subsquares
cannot exceed half the size of their embedding.  In fact, $n\ge
2m$ is also sufficient for the existence of an incomplete latin square of order
$n$ with a hole of order $m$. 

It was shown in \cite{DvB} that a set of $t$-IMOLS$(n;m)$ exist for all sufficiently large integers $n$ and $m$ satisfying $n
\ge 8(t+1)^2m$.  The main tool was a weaker version of Theorem~\ref{main} in which the inequality 
$v > (k-1+\epsilon) w$ was instead $v > k_1 k_2 \cdots k_r w$, where $\alpha(\{k_1,\dots,k_r\}) = \alpha(K)$.
As an application of our result, we can easily improve the the constant factor from quadratic in $t$ to linear in $t$.

\begin{thm}
Let $\epsilon>0$. For $t>t_0(\epsilon)$, there exist $t$-IMOLS$(n;m)$ for all $n,m \ge n_0(t)$ satisfying $n> (1+\epsilon)(t+1)m$.
\end{thm}
 
\begin{proof}
From the prime number theorem for arithmetic progressions, there exists a prime $p \equiv 3\pmod{4}$ in the interval $[t+2,(1+\epsilon/2)(t+1)]$, provided $t \ge t_0(\epsilon)$. Now, let $2^f$ be the smallest power of 2 greater than $t + 1$ and put $K:= \{p,2^f, 2^{f+1}\}$. Observe that an odd prime divisor of $p-1$ cannot divide both $2^f-1$ and $2^{f+1}-1$.  Together with $p-1 \equiv 2 \pmod{4}$, this ensures $\alpha(K)=1$ and $\beta(K)=2$. Then, by Theorem~\ref{main}, there exist IPBD$((n;m), K)$ for all
sufficiently large integers $n,m$ satisfying $n \ge (1+\epsilon)(t+1)m$.
 
Since each $k \in K$ is a prime power exceeding $t+1$, there exist $t$ mutually orthogonal idempotent latin squares of side $k$.
It follows from \cite[Lemma 5.2]{DvB} that there exist $t$-IMOLS$(n;m)$.
\end{proof}
 
\rk
By choosing any $p \equiv 3 \pmod{4}$ exceeding $t+1$, even if it is far from $t+1$, we still have $\min K \le 2^f \le 2(t+1)$, so that
the existence of $t$-IMOLS$(n;m)$ is obtained for sufficiently large $n \ge 2(t+1)m$.  That is, $t_0=1$ is possible when $\epsilon=1$.

\subsection{Closure and subdesigns}

In the preceding subsections, we have used IPBDs as a `template' to construct certain other incomplete designs.  Although it is not our intention to list many more such applications along these lines, we include  a few general remarks and highlights.

Wilson \cite{RMW0} defines a set $K$ of positive integers to be \emph{PBD-closed} if $K=\{n: \exists \, \text{PBD}(n,K)\}$.  A wide variety of designs and related structures can be parameterized by a PBD-closed set; \cite{LW} contains several interesting examples.  For such objects, their existence question effectively reduces to a finite problem.  The method is roughly as follows: (1) find a small (yet sufficiently rich) PBD-closed set $K$ carrying constructible values; (2) verify that $\{n : n\equiv 1 \pmod{\alpha(K)} \text{ and } n(n-1) \equiv 0 \pmod{\beta(K)}\}$ matches the divisibility conditions (if any); and (3) use a PBD$(n,K)$ as a template to build the design for sufficiently large $n$ satisfying the conditions.

We propose that the above method, with IPBD replacing PBD in the last step, is applicable in essentially any situation where the class of designs is PBD-closed with respect to some parameter.  The resulting incomplete designs can accommodate substructures.

\begin{ex}[Resolvable designs with subdesigns]
\label{sub-kts}
A resolvable PBD$(v,\{3\})$ is also known as a \emph{Kirkman triple system}, or KTS$(v)$.  It turns out that Kirkman triple systems are parametrized by a PBD-closed set, though not in the usual way: the set 
$L=\{ n \ge 1 : \exists \, \text{KTS}(2n+1) \}$ is PBD-closed; see \cite{RCW}.  The divisibility conditions on $v$ amount to $v \equiv 3 \pmod{6}$, and this is equivalent to $n \equiv 1 \pmod{3}$.  From an affine plane of order 3 and the solution to Kirkman's schoolgirl problem, we have $\{4,7\} \subseteq L$.  Wilson's theorem delivers the existence of PBD$(n,\{4,7\})$ for sufficiently large $n \equiv 1\pmod{3}$; in fact, existence is known \cite{Brouwer47} for all such positive $n \equiv 1 \pmod{3}$ except $n=10,19$.  In this way, the existence question for KTS$(v)$ can be completely settled using constructions for just four values: $v=9,15,21,39$.

The parameter $n$ in a KTS$(2n+1)$ is the number of parallel classes 
or replication number, and indeed the closure above can be viewed as joining parallel classes from a template PBD$(n,\{4,7\})$.  If instead we use an IPBD$((n;m),\{4,7\})$ as the template, the result is an IPBD$((v;w),\{3\})$, say $(V,W,\mathcal{B})$, where $v=2n+1$ and $w=2m+1$, whose block set $\cB$ can be resolved into $n-m$ full parallel classes and $m$ partial parallel classes (i.e. partitions of $V \setminus W$).  This design accommodates plugging in a KTS$(w)$ on the points of $W$ to produce a KTS$(v)$ containing a sub-KTS$(w)$.  Moreover, parallel classes of the sub-KTS$(w)$ are inherited from those of the KTS$(v)$.

Our main result with $K=\{4,7\}$ yields a construction whenever $v \equiv w \equiv 3 \pmod{6}$ under the assumption $v > (3+\epsilon)w$ and $w > w_0(\epsilon)$.
Using several explicit constructions as ingredients, Rees and Stinson proved the stronger result that there exists a KTS$(v)$ containing a sub KTS$(w)$ if and only if $v \equiv w \equiv 3 \pmod{6}$, $v \ge 3w$.  For details, see \cite[\S 4]{Stinson-KTS} and the references therein. 
\end{ex}

Let $R_k$ denote the set of positive integers $r$ for which there exists a resolvable PBD$(r(k-1)+1,\{k\})$.  Ray-Chaudhuri and Wilson showed in \cite{RCW} that $\{r \ge 1: r \equiv 1 \pmod{k}\} \setminus R_k$ is a finite set.  Extending the above method, we may construct a resolvable PBD$(v,\{k\})$ containing a sub-system of order $w$ if $v \equiv w \equiv k \pmod{k(k-1)}$, $v > (r_1(k)+\epsilon)w$, and $w > w_0(\epsilon)$, where $r_1(k):=\min R_k \setminus \{1\}$.  The value $r_1(k)$ is not known in general; however, the existence of an affine plane of order $k$ implies $r_1(k)=k+1$.  So the existence question for resolvable subdesigns may be nearly settled in the case that the block size $k$ is a prime power.  Even so, infinitely many cases near the boundary $v=kw$ escape the technique, as do finitely many small parameter pairs $(v,w)$.

\begin{ex}[Complementing 3-paths]
\label{3-paths}
Consider an edge-decomposition of a graph into paths of length 3 in such a way that complementing every path produces another such decomposition.  It was shown in \cite{GMR} that the complete graph $K_n$ admits such a `complementing 3-path' decomposition if and only if $n \equiv 1 \pmod{3}$.  In \cite{LW}, Lamken and Wilson modeled such a system as a decomposition of a 2-edge-colored $K_n$ into copies of a 2-edge-colored $K_4$ in which the edges of each color induce a path on 3 edges.  This illustrates that the set $L$ of orders of complete graphs for which such a decomposition exists is PBD-closed.  It is clear that $4 \in L$.  A construction for $n=7$ arises from developing the path $0-1-3-6 \pmod{7}$.
Using the same class of template IPBDs with block sizes $\{4,7\}$ as in the previous example, this is already enough to obtain complementing 3-path decompositions of a \emph{difference} of complete graphs  $K_n \setminus K_m$ when $n \equiv m \equiv 1\pmod{3}$, $n > (3+\epsilon)m$, and $m \ge m_0(\epsilon)$.  An explicit existence result can be found in \cite{RSsub}.  
\end{ex}

We include one more example which provides a new  explicit existence result.

\begin{ex}[Resolvable reverse triple systems]
\label{rrtriple}

A PBD$(v,\{3\})$ is said to be a \emph{reverse triple system} with respect to a particular 
point $p$ when the permutation $\phi_p$ that fixes $p$ but simultaneously  
interchanges all pairs of points $x,y$ for which $\{p,x,y\}$ is a triple of the 
system is, in fact, an automorphism, i.e. leaves the block set invariant. The existence 
of a resolvable reverse triple system on $2n+1$ points is equivalent 
to the decomposition of a 3-edge-colored $K_n$ into copies of two 3-edge-colored 
cliques $K_4$; see \cite[Example 2.10]{LW} for details. 
This decomposition, together with the 
existence of PBD$(n,\{4\})$ for all $n\equiv 1 \  \text{or}\  4 \pmod{12}$, is used to show that 
resolvable reverse triple systems exist for all $v\equiv 3$ or $9 \pmod{24}$. 
Similarly, the existence of IPBD$((n,m),\{4\})$ for $n\equiv 1\  \text{or}\  4 \pmod{12}$, 
$m\equiv 1\  \text{or}\  4 \pmod{12}$ and $n\geq 3m+1$, \cite{RS4hole},  together with 
the same decomposition using 3-edge-colored $K_4$s, provides the existence of 
reverse resolvable triple system on $v=2n+1$ points which contains as a 
subdesign a reverse resolvable triple system on $w=2m+1$ points for all 
$v\equiv 3\  \text{or}\  9 \pmod{24}$ and $w\equiv 3\  \text{or}\  9 \pmod{24}$ satisfying 
$v\geq 3w$. 
\end{ex}

As in \cite{LW}, a variety of designs with additional structure, such as  Whist tournaments, 
Steiner pentagon systems,  Room squares, near resolvable designs, and self-orthogonal latin squares are PBD-closed.  Doubly resolvable BIBDs wth $\lambda =1$ and doubly 
near resolvable BIBDs are also PBD-closed, \cite{Lmors}.
A `near existence theory' for the incomplete variants of all these 
types of designs follows similarly from our work.  In each situation, the remaining small cases and boundary cases must be considered separately, with specialized constructions or non-existence arguments particular to the problem.

\section{Group divisible designs of type $g^n h^1$}
\label{guh1}

\subsection{Existence}

Here, we generalize our result on IPBDs to the setting of $K$-GDDs of type 
$g^nh^1$. The complete result requires several steps.

We  begin by recalling the necessary divisibility conditions from 
Proposition~\ref{neccond-gdd-h} for the existence of a $K$-GDD of type $g^n h^1$, where $gn>0$ are
\begin{eqnarray}
\label{local-guh1}
gn~\equiv~h-g &\equiv& 0 \pmod{\alpha(K)}, ~\text{and}\\
\label{global-guh1}
gn(gn-g+2h)  &\equiv& 0 \pmod{\beta(K)}.
\end{eqnarray}

We also have the following inequality.

\begin{prop}
\label{hole-boundgdd}
Let $k =\min K$. If there  exists a GDD$(g^n h^1,K)$, then   
\begin{eqnarray}
\label{ineq-gdd}
g(n-1) \ge (k-2)h.
\end{eqnarray}
\end{prop}

\begin{proof}
Let $x$ be an element of some group of size $g$. Since $x$ must occur with each 
element in the group of size $h$, the number of blocks which contain $x$ is at least $h$. 
Since $x$ must occur with each of $g(n-1)+h$ elements in blocks of size at least $k$, 
we also have that the maximum number of blocks which contain $x$ is 
$(g(n-1) + h)/(k-1)$. This gives the desired inequality: 
$(g(n-1) + h)/(k-1) \geq h$ or $g(n-1) \geq (k-2)h$. 
\end{proof}

{\sc Remark.}
As in the case of Proposition~\ref{holesize-bound}, equality in (\ref{ineq-gdd}) 
holds if and only if every block intersects the group or `hole' of size $h$ and 
has size $k$. The case of maximal $h$ is constructed from a resolvable 
GDD$(g^n,k-1)$; see \cite[Proposition 6.2]{DLLholes}.

The existence question for GDDs with all but one group of equal size was solved for $K=\{3\}$ in \cite{CHR}.  Even in the case $K=\{4\}$ the problem is still not finished, although after \cite{GL,WG} there remain only a few outstanding cases.

We first use our main result  on IPBDs to provide an existence result  for
$K$-GDDs of type $g^nh^1$ for general $K$ in the case of fixed $g$ and $h$ satisfying $g \mid h$.

\begin{thm}
\label{main-guh1}
Let $g$ and $s$ be positive integers, $K \subseteq \Z_{\ge 2}$, and $\epsilon>0$.  For some $n_0=n_0(K,g,s,\epsilon)$,
a $K$-GDD of type $g^n(gs)^1$ exists for all $n \ge n_0$ and $s$ satisfying
the divisibility conditions with $h=gs$ and $n > (r-2+\epsilon) s$, where $r >1$ is an integer such that there exists a $K$-GDD of type $g^r$.
\end{thm}

\begin{proof}
With $h=gs$, (\ref{local-guh1}) becomes $n \equiv s \equiv 0 \pmod{\alpha(K)/\gcd(g,\alpha(K))}$ and (\ref{global-guh1}) becomes $n(n-1+2s) \equiv  0 \pmod{\beta(K)/\gcd(g^2,\beta(K))}$.  Let $M = \{m \ge 2 : \exists \, K\text{-GDD of type } g^m\}$ and put $r := \min M$.  In his existence theory for uniform group divisible designs, Chang proved in \cite{Chang} that $\alpha(M) = \alpha(K)/\gcd(g,\alpha(K))$ and 
$\beta(M)=\beta(K)/\gcd(g^2,\gamma(K))\gcd(g,\alpha(K))$, or twice this number if it is odd.  We note that
$\gcd(g^2,\gamma) \gcd(g,\alpha)$ divides $\gcd(g^2,\gamma)\gcd(g^2,\alpha) = \gcd(g^2,\beta)$, where here $K$ is suppressed for clarity.  It follows that 
$\beta(K)/\gcd(g^2,\beta(K))$ divides $\beta(M)$. 

Suppose parameters $n$ and $s$ are given satisfying the hypotheses.  From the above calculations, $n$ and $s$ meet the divisibility conditions for IPBD$((n+s;s),M)$, and also the required inequality $n+s > (r-1+\epsilon)s$.  From such an IPBD, blow up each point in this design into a disjoint bundle of $g$ points, and replace each block, say of size $m$, by a $K$-GDD of type $g^m$ on the corresponding points.  The result is a $K$-GDD of type $g^n(gs)^1$.
\end{proof}

Note that the value $r$ in the theorem can be chosen by Theorem~\ref{asym-gdd}, 
or by a direct construction in specific cases. 
In special cases, we may get the desired inequality with $r=k$. In general, we 
can replace  $r$ by $k$ in the necessary inequality of Theorem~\ref{main-guh1}.
by using a slightly longer argument. The proof is very similar to the proof of 
Theorem~\ref{main}.  In this case, we use the template IPBDs,  Theorem~\ref{main-guh1},
to provide the ingredient designs for  Lemma~\ref{ABC}.
 We  sketch some of the details for the proof. 
We write $\gcd(g,\alpha) n/ \alpha = tA_0+B_0$, 
and observe then that $gn = \alpha(tA+B)$, where $g \mid \alpha A, \alpha B$.  Likewise, we put $gs = \alpha(C+z)$, where $g \mid \alpha C, \alpha z$ and $\gamma \mid C$.  Fill groups of our GDD of type $(\alpha A)^t (\alpha B)^1 (\alpha C)^1$ with GDDs of type $g^{\alpha A/g} (\alpha z)^1$ and $g^{\alpha B/g} (\alpha z)^1$, setting aside a common group of size $\alpha z$, noting that these ingredients exist for bounded $z$ and large $A,B$ from 
Theorem~\ref{main-guh1}.   Including the group of size $\alpha z$ with the last group 
or `hole', the resulting GDD has type $g^n(gs)^1$.  The inequality on $n$ and $s$ weakens as needed since $C$ can be as large as $At/(k-2+\epsilon)$ in the lemma. Thus, we have the following result.

\begin{thm}
\label{main-guh2}
Let $g$ and $s$ be positive integers, $K \subseteq \Z_{\ge 2}$, and $\epsilon>0$.  For some $n_0=n_0(K,g,s,\epsilon)$,
a $K$-GDD of type $g^n(gs)^1$ exists for all $n \ge n_0$ and $s$ satisfying
the divisibility conditions with $h=gs$ and $n > (k-2+\epsilon) s$, where $k=\min K$.
\end{thm}

We now turn our attention to the case of any  $h$ where $h \equiv g \pmod{\alpha(K)}$ 
rather than just $h=gs$.  This case is more difficult since 
we first need to construct examples for each possible $h$ and establish the existence 
of $K$-GDDs of type $g^nh^1$ for fixed $h$ and $n$ sufficiently large. There are three 
steps to complete our proof of the general case. We first use `holey group divisible designs', HGDDs, to construct a class of examples for each possible $h$. 
We begin with some preliminary definitions and constructions for HGDDs.

A (uniform) \emph{holey group divisible design} is a quadruple
$(X,\Pi,\Xi,\cB)$, where $X$ is a set of $x$ points, $\Pi$ and $\Xi$ are partitions of 
$X$ and $\cB$ is a collection of subsets of $X$ (blocks) such that 
\begin{itemize}
\item
$\Pi=\{V_1,\dots,V_u\}$ is a partition of $X$ into $u$ \emph{groups} of size
$hm$;
\item
$\Xi=\{W_1,\dots,W_m\}$ is a partition of $X$ into  $m$ \emph{holes} of size $uh$, 
where $|V_i \cap W_j|=h$ for each $i,j$.
\item
$\cB $ is a set of of blocks which meet each group and each hole in at most one point; and
\item
any two points from distinct groups and distinct holes appear together in exactly one block.
\end{itemize}

We abbreviate such a design as an HGDD of type $u \times h^m$.  If each block $B \in \cB$ 
has size from a set $K$, we denote this by HGDD$(u\times h^m, K)$.  When $K=\{k\}$ and 
$h=1$, these designs are also known as \emph{grid designs} or `modified group divisible designs.'
HGDDs are special uniform types of `double group divisible designs' or DGDDs; 
see \cite{GeLingrgdd,GeReesZhu} for further information.

The asymptotic existence of grid designs was established by 
Chang in 1976 in \cite{Chang} (where they are called lattice designs.)  A newer proof can be found in \cite{LW}.

\begin{thm}[Chang, \cite{Chang}]
\label{asymptotic-grid}
Let $v$, $\ell$ be given with $v\geq \ell \geq 2$. There exists a constant $u_0=u_0(v,\ell)$ 
such that an HGDD$(u \times 1^v,\{\ell\}) $ exists for all integers $u \geq u_0$ that 
satisfy 
\begin{eqnarray}
\label{local-grid}
(v-1)(u-1) &\equiv& 0 \pmod{(\ell -1)} ~\text{and}\\
\label{global-grid}
v(v-1)u(u-1) &\equiv& 0 \pmod{\ell (\ell -1)}.
\end{eqnarray}
\end{thm}

Wilson's Fundamental Construction can be used to construct HGDDs from grid designs.

\begin{lemma}
\label{hgdd-class}
Let $\ell$ be a positive integer such that $\ell \equiv 1 \pmod{\alpha(K)}$ and
 $\ell \equiv 0\pmod{\gamma(K)}$. Suppose $u$ is a postive integer, $u\geq u_0$, 
where   $u\equiv 1\pmod{(\ell -1)}$ and $u\equiv 0\pmod{\ell}$, 
then there exists an HGDD$(u\times h^v,K)$.
\end{lemma}

\begin{proof}
We choose $u$ so that there exists an HGDD$(u \times 1^v,\{\ell\})$, by
Theorem~\ref{asymptotic-grid}. Since there exists a GDD$(h^{\ell},K)$, by Theorem~\ref{asym-gdd}, we simply 
apply Wilson's Fundamental Construction with weight $h$ and using the GDDs as ingredient designs 
to get an HGDD$(u\times h^v,K)$.
\end{proof}

Construction 3.20 of \cite{GeLingrgdd} provides a way to construct non-uniform 
GDDs from DGDDs by `filling in groups'. We state the construction in terms of HGDDs.

\begin{cons}[see \cite{GeLingrgdd}]
If there exists an HGDD$(u \times h^v,K)$ and a GDD$(h^{v}a^1,K)$, 
then there is a GDD$((hu)^va^1,K)$.
\end{cons}

We apply this construction with $a=h$ and $v$ chosen so that there exists a 
$K$-GDD of type $h^{v+1}$. If  $v\equiv 0\pmod{\alpha(K)}$ and 
$v\equiv -1 \pmod{\gamma(K)}$, then 
there is a $K$-GDD of type $(hu)^vh^1$ where $u$ is as in Lemma~\ref{hgdd-class}.
The next step is to turn this GDD into a GDD with groups of sizes $g$ and $h$
where $g\equiv h \pmod{\alpha}$. 
To do this, we write $hu=gs$. Notice that the conditions above give us that 
$u\equiv 1 \pmod{\alpha}$ and therefore $s\equiv 1\pmod{\alpha}$. Since 
$u\equiv 0\pmod{\gamma}$, $gs=hu\equiv 0 \pmod{\gamma}$. 
Thus the necessary conditions for the existence of a GDD$(g^s,K)$ 
(Theorem~\ref{asym-gdd}) are satisfied. 
This gives us the first examples of GDDs where $h$ is not a multiple of $g$.

\begin{lemma}
\label{gh-example}
Let  $v\equiv 0\pmod{\alpha(K)}$ and  $v\equiv -1 \pmod{\gamma(K)}$.
There exists a GDD$(g^{sv}h^1,K)$ where $s$ is a positive integer such that 
$s\equiv 1\pmod{\alpha}$ and $gs\equiv 0 \pmod{\gamma}$.
\end{lemma} 

Notice that there is no requirement on $h$ here except that $h\equiv g\pmod{\alpha}$.
These designs can be used together with the same argument used in our sketch of 
the proof of Theorem~\ref{main-guh2} to 
provide a `fixed $h$' result.

\begin{thm}
\label{gh-fixedh}
Let $g$ and $h$ be positive integers  and $K \subseteq \Z_{\ge 2}$. For 
some $n_0=n_0(K,g,h)$, there exists a GDD$(g^nh^1,K)$ for all $n \geq n_0$ and 
$h$ satisfying the divisibility conditions. 
\end{thm}

\begin{proof}
We can use Lemma~\ref{ABC} as above to construct  $K$-GDDs of 
type $g^n(\alpha C + gw)^1$.
In this case, the common group set aside will be of size $gw$ and we do not 
require that $g$ divide $\alpha C$.  We use $K$-GDDs of types $g^{\alpha A/g}(gw)^1$
and $g^{\alpha B/g} (gw)^1$, Theorem~\ref{main-guh2}. 
The existence of these designs 
 implies that  $gw\equiv g \pmod{\alpha}$.  Since we also require $g\equiv h\pmod{\alpha}$, 
this means that $gw\equiv h\pmod{\alpha}$.
So we can choose $C$ and $w$ so that $gsv+h=\alpha C + gw$.
The resulting design is a GDD$(g^n(gsv+h)^1, K)$. The last step is to use 
the designs from Lemma~\ref{gh-example} to fill in the last group. The 
resulting design is a GDD$(g^{n+sv}h^1,K)$.
\end{proof}

Finally, we can apply the same argument once again using Lemma~\ref{ABC} as we did in 
the proof of Theorem~\ref{main} to establish the asymptotic existence of 
GDD$(g^nh^1,K)$.

\begin{thm}
\label{main-guh}
Let $g$ and $h$ be positive integers, $K \subseteq \Z_{\ge 2}$, and $\epsilon>0$.  For some $n_0=n_0(K,g,h,\epsilon)$,
a $K$-GDD of type $g^nh^1$ exists for all $n \ge n_0$ and $h$ satisfying
the divisibility conditions and $g(n-1) > (k-2+\epsilon) h$, where $k=\min K$.
\end{thm}

\begin{proof}
Since the proof is very similar to those above, we just sketch a few essential details. 
Suppose we are given integers $n$, $g$ and $h$  satisfying the necessary divisibility
conditions. Let $g\equiv h\equiv i \pmod{\alpha}$ where $i < \alpha$. 
We write $gn = \alpha (tA +B)$ (as above) and $h=\alpha (C + z) + i$.
In this case we fill in the groups of our GDD of type 
$(\alpha A)^t (\alpha B)^1 (\alpha C)^1$ with GDDs of type $g^{\alpha A/g} (\alpha z+i)^1$ and $g^{\alpha B/g} (\alpha z+i)^1$, setting aside a common group  
of size $\alpha z+i$; these designs exist by Theorem~\ref{gh-fixedh}. The resulting 
GDD has type $g^nh^1$ where $h=\alpha (C+z) +i$.
Since $(\alpha z + i) (k-2) \ll g(\alpha B/g -1) = \alpha B -g$ and  
$C$ can be as large as $At/(k-2+\epsilon)$ in Lemma~\ref{ABC}, we have  the desired 
inequality $(\alpha C + \alpha z + i) (k-2 +\epsilon) <  \alpha (At + B) -g$
or $h (k-2 + \epsilon) < g(n-1)$.
\end{proof}

\subsection{Applications}

As in Section~\ref{ipbd-applications},  $K$-GDDs of type $g^nh^1$ can be used as 
`templates' in Wilson's Fundamental Construction to construct frames and HMOLS (or OPILS) 
with all but one of the groups (holes) the same. 
In this section, we consider a different type of application where, in practice, we want $h$ to 
be fairly small. Group divisible designs of type $g^nh^1$ are quite useful in the 
construction of asymptotically optimal packings and coverings. 
We give only a brief set-up and sketch the main idea, which originates in \cite{WilsonPP}.  The paper \cite{CCLW} can also be consulted
for more details and some related constructions.

A $(v,k,\lam)$-\emph{packing} is a pair $(V,\mathcal{B})$, where $|V|=v$, $\mathcal{B}$ is a family of $k$-subsets of $V$, and with the property that any two distinct elements of $V$ appear together in at most $\lam$ blocks.  A $(v,k,\lam)$-\emph{covering} is defined similarly, but with `at least' in place of `at most'.

Consider, for simplicity, the case of packings with $\lam=1$.  The \emph{leave} of such a packing is the graph $L=(V,E)$, where $\{x,y\} \in E$ if and only if there is no block containing $\{x,y\}$ (and isolated vertices are typically discarded from this graph).
A PBD$(v,k)$ is then a $(v,k,1)$-packing with empty leave. 

The number of blocks $b$ of a $(v,k,1)$-packing satisfies the Johnson bound 
\begin{equation}
\label{johnson-bound}
b \le \left\lfloor \frac{v}{k} \left\lfloor \frac{v-1}{k-1} \right\rfloor \right\rfloor.
\end{equation}

We illustrate our use of GDDs with two simple examples for block size $k=3$.

\begin{ex}
\label{packing-ex}
A $(5,3,1)$-packing can be constructed as two edge-disjoint triangles inside $K_5$; the leave in this case is isomorphic to the four cycle $C_4$. By `filling the hole' of an IPBD$((v;5);\{3\})$ with this example, one obtains an optimal $(v,3,1)$-packing for all $v \equiv 5 \pmod{6}$.  For $v \equiv 0,2 \pmod{6}$, a 3-GDD of type $2^{v/2}$ exists (by deleting a point and all incident blocks from a Steiner triple system of order $v+1$) and furnishes a $(v,3,1)$-packing whose leave is a perfect matching.
\end{ex}

As the preceding example suggests, the groups of a GDD can be filled with packings (or left unfilled) to produce packings with structured leaves.   In particular, using a GDD of type $g^n h^1$ for small $g$ leads to a recursive construction for a congruence class $h \pmod{gk(k-1)}$.

\begin{cons}
\label{packing-cons}
Let $k \ge 2$ and $g$ be an integer with $1 \le i \le k-1$.  
Suppose that, for some $h \equiv g \pmod{k-1}$, there exists an $(h,k,1)$-packing having at least $\lfloor \frac{h}{k} \lfloor \frac{h-1}{k-1} \rfloor \rfloor-e$ blocks.  Then there exists a 
$(v,k,1)$-packing with at least $\lfloor \frac{v}{k} \lfloor \frac{v-1}{k-1} \rfloor \rfloor-e$ blocks for all sufficiently large integers of the form $v \equiv h \pmod{gk(k-1)}$.
\end{cons}

Note that Construction~\ref{packing-cons} results in a leave consisting of several copies of $K_{g}$, a regular graph of degree $g-1$, together with the leave placed on the hole of size $h$.
This can be applied separately for different hole sizes $h$, and different congruence classes $g$ for $v$ modulo $k-1$.
It follows that, for fixed $k$, packings with some maximum deficiency $e$ from the Johnson bound (\ref{johnson-bound}) can be obtained for all sufficiently large integers $v$, provided a certain finite list of packings with maximum deficiency $e$ can be found.

A similar approach can be applied to construct optimal coverings. For example, 
the analogue of Example~\ref{packing-ex} for optimal $(v,3,1)$-coverings uses  4 small optimal coverings with block size $3$ for $v=4,5,6,$ and $8$, 
3-GDDs of types $6^{\frac{v}{6}}$, $6^{\frac{v-4}{6}}4^1$, and 
$6^{\frac{v-8}{6}}8^1$, together with Steiner triple systems and incomplete triple systems IPBD$((v;5),\{3\})$.
This treats all but a finite number of parameters for such coverings. Optimal coverings with block size $4$ 
can also be described using a few small coverings and 4-GDDs of types $6^n$ and 
$6^n15^1$, together with results 
for IPBD$((v,22);\{4\})$ and $(v,4,1)$-BIBDs. References and further information about these 
coverings  can be found in \cite{coverings-handbk}.

These  examples and construction illustrate the idea that 
group divisible designs of the form studied here provide a unified framework for 
constructing both optimal coverings and packings. This framework also extends to the  case of 
general $\lam$. We note  that 
these GDDs are particularly useful in those cases for which `small' explicit optimal 
packings or coverings can be constructed.   We leave a more detailed investigation of these
applications for future studies.

\end{document}